\documentclass[11pt]{article}
\usepackage{fullpage,url,natbib}
\usepackage{times}
\usepackage{amsthm,amsfonts,amsmath,amssymb,epsfig,color,float,graphicx,verbatim}
\usepackage{algorithm}
\usepackage{algpseudocode}
\usepackage{footmisc}
\usepackage{wrapfig}
\usepackage{subcaption}
\usepackage{natbib}
\usepackage{bbm, bm}
\usepackage{enumerate}

\usepackage{hyperref}
\hypersetup{
	colorlinks   = true, 
	urlcolor     = blue, 
	linkcolor    = blue, 
	citecolor   = blue 
}

\newtheorem{theorem}{Theorem}
\newtheorem*{theorem*}{Theorem}

\newtheorem*{proposition*}{Proposition}
\newtheorem{example}[theorem]{Example}
\newtheorem{lemma}[theorem]{Lemma}

\newtheorem{definition}[theorem]{Definition}
\newtheorem{remark}[theorem]{Remark}

\newcommand{\one}{\mathbbm{1}}

\newcommand{\reals}{\mathbb{R}}

\newcommand{\Ecal}{\mathcal{E}}

\newcommand{\bmeta}{{\bm{\eta}}}

\newcommand{\Proj}{\mathrm{Proj}}
\newcommand{\dist}{\mathrm{dist}}

\newcommand{\bzero}{\boldsymbol{0}}

\usepackage{xcolor}

\newcommand{\beq}{\begin{eqnarray*}}
\newcommand{\eeq}{\end{eqnarray*}}
\newcommand{\beqn}{\begin{eqnarray}}
\newcommand{\eeqn}{\end{eqnarray}}

\usepackage{ifmtarg}
\usepackage{xifthen}%
\newcommand{\ent}[1][]{%
\ifthenelse{\isempty{#1}}{%
\mathrm{H}
}{
\mathrm{H}^{(#1)}
}}

\newcommand{\loch}[1][]{%
\ifthenelse{\isempty{#1}}{%
\mathrm{h}
}{
\mathrm{h}^{(#1)}
}}

\newcommand{\NN}{\mathbb{N}}

\newcommand{\half}{\frac{1}{2}}

\newcommand{\E}{\mathbb{E}}

\newcommand{\BB}{\mathbb{B}}

\newcommand{\be}{\mathbf{e}}
\newcommand{\bx}{\mathbf{x}}

\newcommand{\bg}{\mathbf{g}}

\newcommand{\bv}{\mathbf{v}}

\newcommand{\by}{\mathbf{y}}

\newcommand{\Xcal}{\mathcal{X}}

\newcommand{\norm}[1]{\left\|#1\right\|}
\newcommand{\inner}[1]{\left\langle#1\right\rangle}

\title{Gradient Descent’s Last Iterate is Often (slightly) Suboptimal}

\author{
Guy Kornowski$^{1,2}$
\qquad
Ohad Shamir$^{2,3}$
\vspace{3pt}
\\
$^{1}$Institute for Advanced Study
\\
$^{2}$Weizmann Institute of Science
\\
$^{3}$University of Toronto and the Vector Institute
}

\begin{document}

\maketitle

\begin{abstract}
We consider the well-studied setting of minimizing a convex Lipschitz function using either gradient descent (GD) or its stochastic variant (SGD), and examine the last iterate convergence. We prove that $\Theta(\sqrt{\log(T)/T})$ is the optimal anytime rate attainable with a fixed stepsize schedule (that does not depend on the total number of steps $T$). The upper bound holds already for SGD, and improves on the $\log(T)/\sqrt{T}$ rate of \citet{shamir2013stochastic}, as well as the $\sqrt{\log(T)/T}$ rate of \citet{zamani2023exact} which was proven only for deterministic GD. The matching lower bound holds already for GD, and confirms the conjecture of \citet{jain2019making}, namely that the optimal $1/\sqrt{T}$ rate for the last iterate cannot be obtained without prior knowledge of $T$. 
\end{abstract}

\section{Introduction}

Let $f:\Xcal\to\reals$ be a convex Lipschitz function over a convex domain $\Xcal\subset\reals^d$, and consider the stochastic gradient descent (SGD) algorithm\footnote{Strictly speaking, in this work we do not assume differentiability of $f$, and consider \emph{sub}-gradient methods.} starting from $\bx_0\in\Xcal:$
\begin{equation} \label{eq: GD update rule}
    \bx_{t+1}=\Proj_{\Xcal}\left(\bx_t-\eta_t \bg_t\right)
    ~~\text{for all}~~t\in\NN\cup\{0\}~,
\end{equation}
where $\E[\bg_t\,|\,\bx_t]\in\partial f(\bx_t)$, and $(\eta_0,\eta_1,\dots)$ is the stepsize sequence (also referred to as learning rate).

Although SGD has been studied extensively for decades \citep{robbins1951stochastic,nemirovskiyudin1983}, and is commonly considered the main workhorse of machine learning \citep{zhang2004solving,bottou2010large,goodfellow2016deep},
there still remain some surprisingly fundamental gaps in its theoretical analysis.

The textbook analysis of SGD implies that for stepsizes $\eta_{t}=\Theta(1/\sqrt{t})$, the error of the \emph{average} iterate ($\frac{1}{T}\sum_{t=1}^{T}\bx_t$ after $T$ steps) converges at a $1/\sqrt{T}$ rate, which is also the best possible among all gradient-based methods (cf. \citealp[Theorem 3.4]{hazan2016introduction} and \citet{nemirovskiyudin1983}). This well-known result has the clear drawback that in practice, it is much more common to return the \emph{last} iterate obtained, namely $\bx_T$, rather than the average iterate.
However, the last iterate of SGD with the aforementioned stepsizes only converges at a suboptimal $\log (T)/\sqrt{T}$ rate \citep{shamir2013stochastic,harvey2019tight}, differing from the information-theoretically optimal $1/\sqrt{T}$ rate by an excess log factor.

This unsatisfying aspect led to substantial research efforts analyzing the last iterate convergence of SGD under various settings, as well as motivated the design of modified output rules that recover the optimal rate. We further discuss prior works along these lines in Section~\ref{sec: related}.

Ultimately, \citet{jain2019making} proved that the last iterate of SGD, with suitable stepsizes, can indeed converge at an optimal worst-case $1/\sqrt{T}$ rate after $T$ steps.
At the heart of this remarkable result is the design of a non-trivial stepsize sequence $(\eta_{t})_{t=0}^{T-1}$, which unfortunately still suffers from a shortcoming: Its values depend on $T$, and therefore the number of steps $T$ must be chosen in advance.
In other words, after completing a predetermined budget of $T$ steps, if additional steps are required or desired, it is unclear how to continue minimizing the error at an optimal rate.
However, in many applications, it is desirable to have  algorithms with \emph{anytime} guarantees, and the anytime nature of SGD with standard stepsizes is arguably one of its appeals in the first place. This leads to an undesirable mismatch between theory and practice, which affects the practical design of learning rate schedules without prior knowledge of the time horizon (see for  example \citet{defazio2024road}). 

The described issue was raised already by \citet{jain2019making}, where the authors conjectured that
\begin{quote}
\textit{``in absence of a priori information about $T$, no step-size sequence can ensure the information theoretically optimal error rates for the final iterate of SGD.''}
\end{quote}

In other words, the conjecture states that in order to converge at an optimal rate, SGD requires either averaging, or knowing $T$ in advance.
\citeauthor{jain2019making} provided some evidence for the validity of their conjecture by considering noisy gradient oracles and their probabilistic consequences, however, the conjecture has so far remained unproven.
This leaves a natural open problem in the theoretical analysis of SGD, as to whether the last iterate can achieve optimal convergence in an anytime fashion.

In this work, we resolve this question in the negative, and moreover, show that $\Theta(\sqrt{\log(T)/T})$ is the precise anytime rate attainable for the last iterate of SGD, by any fixed stepsize schedule not depending on $T$. In fact, we show that the implied $\Omega(\sqrt{\log(T)/T})$ lower bound holds even for deterministic GD (where there is no noise in the gradients), for which the lower bound for SGD follows as a special case. As to the $O(\sqrt{\log(T)/T})$ upper bound applicable to SGD, it strictly improves on the $O(\log(T)/\sqrt{T})$ upper bound proven in \citet{shamir2013stochastic}, as well as the $O(\sqrt{\log(T)/T})$ of \citet{zamani2023exact} which applies only to GD. 

The paper is structured as follows. After discussing related work in Section~\ref{sec: related}, we present in Section~\ref{sec: prelim} the formal setting that we consider. In Section~\ref{sec: main result} we state our main results, and discuss some notable consequences.
In Section~\ref{sec: proof} we present the proof, besides some technical details which are deferred to the appendix. We conclude in Section~\ref{sec: conclusion}.

\section{Related Work} \label{sec: related}

\paragraph{Last iterate of SGD.}
As previously discussed, classical analyses of SGD typically deal with the convergence of the average iterate \citep{polyak1992acceleration,nemirovski2009robust}. \citet{zhang2004solving}  analyzed the convergence rate of the last iterate of SGD with constant stepsize for  learning linear predictors. This was extended to general convex objectives and decaying stepsizes in \citet{shamir2013stochastic}. \citet{harvey2019tight} further established tight high probability bounds.
Tight constants were derived for GD in the deterministic setting by \citet{zamani2023exact}.
The general smooth setting was studied by \citet{moulines2011non}, with results later improved by \citet{taylor2019stochastic,liu2024revisiting}. Several recent works studied last iterate convergence in the so called (near-)interpolation or low-noise regime, first for least-squares \citep{varre2021last} and subsequently for general smooth losses \citep{attia2025fast,garrigos2025last}.
Several works considered objectives with an empirical risk minimization structure, and established last iterate convergence for SGD with respect to different sampling schemes \citep{gower2019sgd,liu2024last},
which was further studied through connections to continual learning \citep{evron2025continual,cai2025last}.

\paragraph{Optimal convergences rates.}
Several works considered modifications of SGD or of its output rule in order to remove excess log factors and recover optimal rates for strongly-convex objectives \citep{hazan2014beyond,rakhlin2012making,lacoste2012simpler}. The question of whether some form of averaging is needed for SGD to recover optimal rates in this context was raised by \citet{shamir2012open}, and answered affirmatively by \citet{harvey2019tight} both for strongly-convex and convex objectives. However, their results applied only for the theoretically standard choice of step sizes ($\Theta(1/\sqrt{t})$ in the convex case and $\Theta(1/t)$ in the strongly-convex case). As previously discussed, \citet{jain2019making} then showed that it is possible to recover optimal rates with SGD's last iterate by considering non-standard stepsizes that depend on the stopping time $T$.

\paragraph{Anytime smooth GD.}
In the related setting of smooth convex deterministic optimization, recent works designed non-standard stepsizes for GD that accelerate the well-established convergence rates of constant stepsizes \citep{altschuler2024acceleration,grimmer2025accelerated}. It was noted by \citet{kornowski2024open} that these results do not yield anytime improvements, and it was asked whether such improvements are possible, which was then
answered positively
by \citet{zhang2025anytime}. It is interesting to note that in the smooth deterministic setting, the best known anytime bounds differ from non-anytime bounds by polynomial factors, and it is an open problem as to whether this is necessary. 
In contrast, in this work we prove that poly-logarithmic gaps are necessary in the non-smooth setting.

\section{Preliminaries} \label{sec: prelim}

\paragraph{Notation.}
We denote $\NN:=\{1,2,\dots\},~\NN_0:=\{0,1,2,\dots\},~\reals_{\geq 0}:=\{x\in\reals:x\geq 0\}$ and $[n]:=\{1,\dots,n\}$.
We use bold-faced font to denote vectors, e.g., $\bx\in\reals^d$, and denote by $\|\cdot\|$ the Euclidean norm.
We denote $\dist(\bx,A):=\inf_{\mathbf{a}\in A}\|\bx-\mathbf{a}\|$ for $A\subset\reals^d,~\bx\in\reals^d$,
and by $\Proj_{A}$ the projection operator onto $A$.
$\partial f(\cdot)$ denotes the sub-gradient set of a convex function $f$.
A function $f:\Xcal\to\reals$ is called $G$-Lipschitz if for any $\bx,\by\in\Xcal:|f(\bx)-f(\by)|\leq G\norm{\bx-\by}$, or equivalently for a convex function, for any $\bx\in\Xcal$ and $\bg\in\partial f(\bx): \|\bg\|\leq G$.
We use the standard big-O and little-o notation, with $O(\cdot),~\Omega(\cdot)$ and $\Theta(\cdot)$ hiding absolute constants that do not depend on problem parameters, and $a_t=o(b_t)$ meaning $\lim_{t\to\infty}a_t/b_t=0$.

\paragraph{Setting.}

Throughout the paper we impose the standard assumptions that $f:\Xcal\to\reals$ is convex, and $\Xcal\subset\reals^d$ is a closed convex domain of Euclidean diameter at most $D$. Given an initial point $\bx_0\in\Xcal$,
we consider SGD (as in Eq.~(\ref{eq: GD update rule})), where $\bg_t$ satisfies 
$\E[\bg_t|\bx_t]\in \partial f(x_t)$ and , and $\E[\|g_t\|_2^2|\bx_t]\leq G^2$ almost surely for some $G>0$. For deterministic GD, we assume $\bg_t\in\partial f(\bx_t)$. 

Given a fixed stepsize sequence $\bmeta=(\eta_t)_{t=0}^{\infty}\in\reals_{\geq0}^{\NN_0}$, we are interested in its worst-case anytime last iterate guarantee:

\begin{definition}
For stepsize sequence $\bmeta=(\eta_{t})_{t=0}^{\infty}$, we say that $\Ecal^{G,D}_\bmeta:\NN\to\reals_{\geq 0}$ is an \emph{anytime convergence rate guarantee} for $\bmeta$, if for all convex functions $f$, domains $\Xcal$ and vectors $\bg_0,\bg_1,\ldots$ satisfying the assumptions in the previous paragraph, and any initialization $\bx_0\in\Xcal$,
the iterates $\bx_1,\bx_2,\ldots$ generated by SGD satisfy
\[
\E[f(\bx_t)]-\min_{\bx\in\Xcal} f(\bx) \leq  \Ecal^{G,D}_{\bmeta}(t)~~~\text{for all}~~~t\in\NN~.
\]
\end{definition}

For example, we recall the following known anytime last iterate convergence guarantee:

\begin{example}[\citep{shamir2013stochastic}]
\label{ex: known upper}
The stepsize sequence $\bmeta=(D/G\sqrt{t+1})_{t=0}^{\infty}$ has the anytime convergence guarantee
\[
\Ecal_\bmeta^{G,D}(t)=\frac{DG(4+2\log t)}{\sqrt{t}}~.
\]
\end{example}

\section{Main Results} \label{sec: main result}

We are now ready to present our main result:

\begin{theorem} \label{thm: main}
No stepsize sequence $\bmeta\in\reals_{\geq0}^{\NN_0}$ has an anytime convergence guarantee satisfying
\[
\Ecal^{G,D}_\bmeta(t)=o\left(DG\sqrt{\frac{\log(t)}{t}}\right)~~\text{as $t\rightarrow \infty$}~.
\]
In particular, neither SGD nor GD's last iterate can yield the optimal $1/\sqrt{t}$ rate in an anytime fashion.
\end{theorem}

\begin{remark}[Absence of noise]
    As previously mentioned, Theorem~\ref{thm: main} confirms the conjecture of \citet{jain2019making}, and is in fact stronger as it applies even in the noiseless case of GD. It is interesting to note that an intuition conveyed by \citet{jain2019making} is that noise in the sub-gradients can lead SGD to be ``bad in expectation'' (precisely defined therein) even in one dimension.
Our proof is based on a different perspective, where instead of noise, the key factors driving the lower bound constructions are high-dimensionality, and the stepsize sequence
not ``knowing'' in advance when the algorithm should stop. We discuss this in more detail in Section~\ref{sec: proof}.
\end{remark}

Before delving into the analysis in the next section, we present here the high-level idea behind the proof of Theorem~\ref{thm: main}. As mentioned in the introduction, \citet{harvey2019tight} showed that any particular stepsize schedule of the form $\eta_t= c/\sqrt{t}$ (with $c>0$ a fixed absolute constant) leads to an extra log factor in the last iterate error. We therefore aim at showing that in a sense, any successful stepsize sequence cannot be too far from this form, and thus leads to a similar (albeit  slightly looser) bound. Specifically, in Lemmas~\ref{lem: 1d bound 1} and \ref{lem: 1d bound 2} we establish that stepsizes that lead to $\tilde{O}(1/\sqrt{t})$ convergence, up to at most a log factor, must satisfy on one hand that $\eta_t= \tilde{O}(1/\sqrt{t})$ at every step, but also that $\sum_{j=0}^{t}\eta_j=\tilde{\Omega}(\sqrt{t})$ for every cumulative sum. Thus, by examining sufficiently long time intervals, we can deduce that any stepsize sequence $\bmeta$ that leads to $\tilde{O}(1/\sqrt{t})$ last iterate error, cannot be ``too far'' from $\eta_t\approx 1/\sqrt{t}$ up to log factors at ``most'' time indices. Then, in Lemma~\ref{lem: high dim bound} we provide a robust extension of the lower bound construction used in the analysis of \citet{harvey2019tight}, which instead of assuming that $\eta_t= c/\sqrt{t}$ exactly, depends directly on the values of $\bmeta$. By combining these pieces, further algebraic manipulations result in our lower bound.

Next, we provide a matching upper bound, applicable to SGD. 
\begin{theorem} \label{thm: upper}
The stepsize sequence $\eta_t=\frac{D}{G\sqrt{(t+1)(1+\log(t+1))}}$ has the anytime convergence guarantee
\[
\E[f(\bx_t)]-\min_{\bx\in\Xcal}f(\bx)\leq
3DG\sqrt{\frac{1+\log(t+1)}{t+1}}
~~~\text{for all}~~~t\in\NN~.
\]
\end{theorem}

A similar $O(\sqrt{\log(T)/T})$ rate was proven in \citep{zamani2023exact}, for deterministic GD. The proof refines the suffix-averaging argument of \citet{shamir2013stochastic} by comparing suffix averages only over the last half of the trajectory, where all stepsizes are within a constant factor of $\eta_t$, which reduces the accumulated error to $O(G^2\eta_t\log t+D^2/(t\eta_t))$, balanced by $\eta_t\propto1/\sqrt{t\log t}$. 

\begin{proof}[Proof of Theorem~\ref{thm: upper}]
Let $\bx^*\in\arg\min_{\bx\in\Xcal}f(\bx)$.
For $t=0$ the claim is immediate: by convexity, $f(\bx_0)-f(\bx^*)\leq\inner{\E[\bg_0],\bx_0-\bx^*}\leq GD$, since $\E[\bg_0]\in\partial f(\bx_0)$ and $\|\E[\bg_0]\|\leq\sqrt{\E\|\bg_0\|^2}\leq G$. Hence, fix $t\geq 1$.

The standard analysis of GD shows that for any $j$ and any $\bx\in\Xcal$ which is determined by $\bg_0,\dots,\bg_{j-1}$ (in particular, any fixed $\bx$, or $\bx=\bx_k$ for $k\leq j$):
\begin{align*}
\E\left[\|\bx_{j+1}-\bx\|^2\right]&\leq \E\left[\|\bx_{j}-\eta_j\bg_j-\bx\|^2\right]\\
&=\E\left[\|\bx_{j}-\bx\|^2\right]-2\eta_j\E\left[\inner{\bg_j,\bx_j-\bx}\right]+\eta_j^2\E\left[\|\bg_j\|^2\right]\\
&\leq \E\left[\|\bx_{j}-\bx\|^2\right] -2\eta_j\E\left[f(\bx_j)-f(\bx)\right]+\eta_j^2 G^2~,
\end{align*}
where the first inequality is by nonexpansiveness of the projection onto $\Xcal$, and the last inequality uses $\E[\bg_j\,|\,\bg_0,\dots,\bg_{j-1}]\in\partial f(\bx_j)$ together with convexity. Rearranging and summing over $j=k,\dots,t$ yields, for any $\bx\in\Xcal$ determined by $\bg_0,\dots,\bg_{k-1}$,
\begin{equation} \label{eq: GD jk}
2\sum_{j=k}^{t}\eta_j\E\left[f(\bx_j)-f(\bx)\right]\leq \E\left[\|\bx_k-\bx\|^2\right]+G^2\sum_{j=k}^{t}\eta_j^2~.
\end{equation}

Denote $S_k:=\sum_{j=k}^{t}\eta_j,~A_k:=\frac{1}{S_k}\sum_{j=k}^{t}\eta_j \E[f(\bx_j)]$.
Since $\eta_j$ is decreasing, and for any $\lfloor t/2\rfloor\leq j\leq t$ we have $\frac{t+1}{j+1}\leq 2$ and
$\frac{1+\log(t+1)}{1+\log(j+1)}\leq 1+\log 2$, it holds that $\eta_j\in[\eta_t, \sqrt{2(1+\log 2)}\,\eta_t]\subseteq[\eta_t,2\eta_t]$. Consequently, for any $k\geq\lfloor t/2\rfloor:~\sum_{j=k}^{t}\eta_j^2\leq 2\eta_t S_k$.
Thus, applying Eq.~\eqref{eq: GD jk} with $k=\lfloor t/2\rfloor,~\bx=\bx^*$ gives
\begin{align}
&2S_{\lfloor t/2\rfloor} \left(A_{\lfloor t/2\rfloor}-{\min}_{\bx\in\Xcal} f(\bx) \right)\leq
D^2+G^2\sum_{j=\lfloor t/2\rfloor}^{t}\eta_j^2
\leq D^2+2G^2\eta_t S_{\lfloor t/2\rfloor}
\nonumber\\
\implies~~~
&A_{\lfloor t/2\rfloor}-{\min}_{\bx\in\Xcal} f(\bx)\leq \frac{D^2}{2S_{\lfloor t/2\rfloor}}+G^2\eta_t~.\label{eq: A- ineq}
\end{align}
Moreover, applying Eq.~\eqref{eq: GD jk} with $\bx=\bx_k$ for $\lfloor t/2\rfloor \leq k<t$ leads to
\begin{align*}
A_{k+1}-A_k
&=\frac{\eta_k}{S_k S_{k+1}}\sum_{j=k}^{t}\eta_j\E[f(\bx_j)-f(\bx_k)]
\leq
\frac{\eta_k}{S_k S_{k+1}} \cdot \frac{G^2}{2}\sum_{j=k}^{t}\eta_j^2
\leq \frac{\eta_k}{S_k S_{k+1}} \cdot G^2\eta_t S_k
=\frac{G^2\eta_k \eta_t}{S_{k+1}}
\\&\leq \frac{2G^2\eta_t}{t-k}~,
\end{align*}
where the last inequality holds because $\eta_k\leq 2\eta_t$ and $S_{k+1}=\sum_{j=k+1}^{t}\eta_j\geq (t-k)\eta_t$.
Summing the inequality above over $k=\lfloor t/2\rfloor,\dots,t-1$, using $A_t=\E[f(\bx_t)]$, Eq.~\eqref{eq: A- ineq}, and $S_{\lfloor t/2\rfloor}\geq (t-\lfloor t/2\rfloor+1)\eta_t\geq\frac{t+1}{2}\eta_t$, we get
\begin{align*}
\E[f(\bx_t)]-\min_{\bx\in\Xcal}f(\bx)
\leq \frac{D^2}{(t+1)\eta_t}+G^2\eta_t\left(1+2H_{t-\lfloor t/2\rfloor}\right)
\leq \frac{D^2}{(t+1)\eta_t}+2G^2\eta_t\left(1+\log(t+1)\right)
~,
\end{align*}
where $H_s:=\sum_{j=1}^{s}\frac{1}{j}$ is the harmonic sum, and we used $t-\lfloor t/2\rfloor\leq\frac{t+1}{2}$ and $H_s\leq 1+\log s$ to obtain $1+2H_{t-\lfloor t/2\rfloor}\leq 3+2\log(t+1)-2\log 2\leq 2(1+\log(t+1))$.
Finally, substituting the value of $\eta_t$, we obtain
\[
\E[f(\bx_t)]-\min_{\bx\in\Xcal}f(\bx)
\leq 3DG\sqrt{\frac{1+\log(t+1)}{t+1}}~.
\]
\end{proof}

\section{Proof of Theorem~\ref{thm: main}} \label{sec: proof}

In this section, we provide the core arguments that are used to prove Theorem~\ref{thm: main}. Additional technical proofs of lemmas are deferred to the appendix.

First, by a rescaling argument, which replaces any $G$-Lipschitz function $f$ over any domain $\Xcal$ of diameter $D$ by $\tilde{f}(\bx)=\frac{2}{GD}f(\frac{D}{2}\bx)$ (which is $1$-Lipschitz over the domain $\frac{2}{D}\Xcal$ of diameter $2$), we see that it suffices to prove Theorem~\ref{thm: main} for $G=1,\,D=2$. We abbreviate notation accordingly by denoting $\Ecal_\bmeta=\Ecal^{1,2}_\bmeta$.

Let $\bmeta=(\eta_t)_{t=0}^{\infty}$ be some stepsize sequence, and suppose it has a corresponding anytime guarantee $\Ecal_\bmeta$ satisfying
\begin{equation}
\Ecal_\bmeta(t)\leq \frac{\phi(t)}{\sqrt{t}}~,
\label{eq: def phi}
~~
\text{where}~~\phi:\NN\to[1,\infty)\text{~is non-decreasing.}
\nonumber
\end{equation}

For instance, $\phi$ might be of the form $\phi(t)=c_1\log^{c_2}(t)+c_3$ for some constants $c_1,c_2,c_3>0$. Our goal is to show that it \emph{cannot} be that $\phi(t)=o(\sqrt{\log(t)})$.
Note that there is no loss in assuming $\phi$ is non-decreasing: if $\Ecal_\bmeta(t)=o(\sqrt{\log(t)/t})$, then $\phi(t):=\max\{1,\max_{1\leq s\leq t}\sqrt{s}\Ecal_\bmeta(s)\}$ also satisfies $\phi(t)=o(\sqrt{\log(t)})$.

We start by establishing basic lower bounds on $\Ecal_\bmeta(t)$, and therefore on $\phi(t)$ (which is at least $\Ecal_\bmeta(t)\sqrt{t})$,
in terms of the stepsizes $(\eta_0,\dots,\eta_{t-1})$. The first such bound shows that large steps can result in large last iterate error (in the worst-case), in case the algorithm happens to stop right after a large step. Formally, we show the following:

\begin{lemma} \label{lem: 1d bound 1}
     For any stepsize sequence $\bmeta$ and $t\in\NN$, it holds that
     \[
     \Ecal_\bmeta(t)\geq \min\{\eta_{t-1},1\}~.
     \]
\end{lemma}

\begin{proof}[Proof of Lemma~\ref{lem: 1d bound 1}]

Given $\bmeta$ and $t$, let $\epsilon>0$ be some arbitrarily small number so that $\epsilon<\min\{1,\,\eta_{t-1},\,\sum_{j=0}^{t-2}\eta_j\}$,
and denote $c_\epsilon:=\frac{\epsilon}{\sum_{j=0}^{t-2}\eta_j}$. Consider the univariate function over $[-1,1]:$
\[
f(x)=\begin{cases}
-x,&\text{if }x< 0
\\
c_\epsilon x, &\text{if }0\leq x \leq \epsilon
\\
x-\epsilon+c_\epsilon\epsilon, &\text{if } x >\epsilon
\end{cases}
~.
\]
Note that $f$ is $1$-Lipschitz, convex since $c_\epsilon<1$, and minimized at $f(0)=0$. Consider the iterates of GD initialized at $x_0=\epsilon$. It is easy to see that for $i<t-1$ it holds that $x_{i+1}=x_i-\eta_i c_\epsilon$, and therefore
\[
x_{t-1}=\epsilon-c_\epsilon\sum_{j=0}^{t-2}\eta_j=0~.
\]
Suppose that at time $t-1$ the sub-gradient is given by $-1\in\partial f(x_{t-1})$, thus resulting in $x_{t}=\Proj_{[-1,1]}(\eta_{t-1})=\min\{\eta_{t-1},1\}$.
Noting that $\eta_{t-1}>\epsilon$, we get that $\Ecal_\bmeta(t)\geq f(x_{t})=\min\{\eta_{t-1},1\}-\epsilon+c_\epsilon\epsilon$. Since this inequality holds for arbitrarily small $\epsilon$, it must hold that
\[
\Ecal_\bmeta(t)\geq \lim_{\epsilon\to0^+}(\min\{\eta_{t-1},1\}-\epsilon+c_\epsilon \epsilon)=\min\{\eta_{t-1},1\}
~.
\]

\end{proof}

Second, we prove a lower bound inversely proportional to the accumulated stepsize sum:

\begin{lemma} \label{lem: 1d bound 2}
    For any stepsize sequence $\bmeta$ and $t\in\NN$ such that $\sum_{j=0}^{t-1}\eta_{j}\geq  \half$, it holds that
    \[
    \Ecal_\bmeta(t)\geq \frac{1}{4e^{2}\sum_{j=0}^{t-1}\eta_j}~.
    \]
\end{lemma}

\begin{proof}[Proof of Lemma~\ref{lem: 1d bound 2}]

The proof closely follows the proof of \citet[Theorem 1]{kornowski2024open}. Given $\bmeta$ and $t$, consider the convex quadratic
\[
f(x)=\frac{x^2}{4\sum_{j=0}^{t-1}\eta_j}~.
\]
Note that $f$ is $1$-Lipschitz over $[-1,1]$ as long as $\sum_{j=0}^{t-1}\eta_{j}\geq \half$, and is minimized at $f(0)=0$.
Examining the iterates of GD initialized at $x_0=1$, we get that
\[
x_{i+1}
=x_i-\eta_i\cdot\nabla f(x_i)
=x_i-\eta_i\cdot\frac{x_i}{2\sum_{j=0}^{t-1}\eta_j}
=\left(1-\frac{\eta_i}{2\sum_{j=0}^{t-1}\eta_j}\right)x_i~,
\]
and thus a simple induction reveals that
\[
x_t
=\prod_{j=0}^{t-1}\left(1-\frac{\eta_j}{2\sum_{j=0}^{t-1}\eta_j}\right)~.
\]
Therefore, using the simple inequality $\log(1-z)\geq -2z$ for any $z\in[0,\half]$, we get
\begin{align*}
\Ecal_\bmeta(t)
\geq f(x_t)
&=\frac{1}{4\sum_{j=0}^{t-1}\eta_j}\cdot\prod_{j=0}^{t-1}\left(1-\frac{\eta_j}{2\sum_{j=0}^{t-1}\eta_j}\right)^2
\\&=\frac{1}{4\sum_{j=0}^{t-1}\eta_j}\cdot\exp\left[2\cdot\sum_{j=0}^{t-1}\log\left(1-\frac{\eta_j}{2\sum_{j=0}^{t-1}\eta_j}\right)\right]
\\&\geq \frac{1}{4\sum_{j=0}^{t-1}\eta_j}\cdot\exp\left[-4\cdot\sum_{j=0}^{t-1}\frac{\eta_j}{2\sum_{j=0}^{t-1}\eta_j}\right]
\\&=\frac{e^{-2}}{4\sum_{j=0}^{t-1}\eta_j}~.
\end{align*}

\end{proof}

By considering both Lemma~\ref{lem: 1d bound 1} and \ref{lem: 1d bound 2}, we see that in order to achieve low anytime last iterate error, each individual step cannot be too large, yet the step sum has to grow at roughly a $\sqrt{t}$ rate up to possible log factors.
Next, we establish a lower bound on $\Ecal_\bmeta$ over an interval, directly in terms of stepsizes along that interval.

\begin{lemma} \label{lem: high dim bound}
For any stepsize sequence $\bmeta$, any $t_1\leq t_2$, and $\delta>0$ such that $\max_{t_1\leq j\leq t_2}\eta_j\leq\delta$, it holds that
\[
\max_{t_1+1\leq t\leq t_2+1}\Ecal_\bmeta(t)\geq
\frac{\sum_{j=t_1}^{t_2}\eta_j}{64(t_2-t_1+1)(1+\delta\sum_{j=t_1}^{t_2}\eta_j)}\log\left(1+\frac{1}{2\delta}\sum_{j=t_1}^{t_2}\eta_j\right)~.
\]
\end{lemma}

The proof of Lemma~\ref{lem: high dim bound} is more involved than the previously discussed bounds. It is high-dimensional in nature, and is based on modifying a technique of \citet{harvey2019tight} which constructs a lower bound for max-of-linear functions in the style of \citet{nemirovskiyudin1983}.
To do so, we carefully construct a sequence of linear functions $\bx\mapsto \bv_i^\top \bx$ that depends both on the stepsize sequence $\bmeta$ and on $\phi$, such that applying GD to $\bx\mapsto \max_{i}\{\bv_i^\top \bx\}$ results in an analytically solvable trajectory (Eq.~\eqref{eq: x_t,j} in the appendix), with a last iterate error that we can lower bound. 
The full details appear in the appendix.

Following Lemma~\ref{lem: high dim bound}, let $T\in\NN$ be sufficiently large, and denote $S:=\sum_{t=T}^{2T-1}\eta_t,~\delta:=\frac{\phi(2T)}{\sqrt{T}}$. Since $\delta<1$ for sufficiently large $T$, Lemma~\ref{lem: 1d bound 1} implies that for all $T\leq t< 2T:~\eta_t\leq \Ecal_\bmeta(t+1)\leq\frac{\phi(t+1)}{\sqrt{t}}
\leq\delta$. Therefore, Lemma~\ref{lem: high dim bound} gives
\begin{equation} \label{eq: delta>=Slog}
\delta\geq \frac{S}{64 T(1+\delta S)}\log\left(1+\frac{S}{2\delta}\right)~.    
\end{equation}

Assume towards contradiction that $\phi(t)=o(\sqrt{\log(t)})$.
If $S\geq 1$, then $\log(1+S/(2\delta))=\log(1+S\cdot\sqrt{T}/(2\phi(2T)))
\geq
\frac{1}{4}\log T$ for sufficiently large $T$, and therefore Eq.~\eqref{eq: delta>=Slog} gives $\delta\geq \frac{S\log(T)}{256 T(1+\delta S)}$. Plugging back the definition of $\delta$ and rearranging yields
\begin{align*}
S(\log(T)-256 \phi^2(2T))\leq 256 \sqrt{T}\phi(2T)~.
\end{align*}
Since $\phi^2(2T)=o(\log(T))$, it follows that for sufficiently large $T$,
\begin{equation} \label{eq: S<=}
\sum_{j=T}^{2T-1}\eta_j=S\leq \frac{512\sqrt{T}\phi(2T)}{\log(T)}~.    
\end{equation}
This bound also clearly holds when $S<1$ for large $T$ (since the right-hand side goes to infinity), and so we obtain Eq.~\eqref{eq: S<=} in either case.

Now, we sum Eq.~\eqref{eq: S<=} over dyadic intervals.
Fixing large $k_0$, and letting $T=2^m$ be arbitrarily large, using the monotonicity of $\phi$ gives
\begin{align*}
\sum_{j=0}^{T-1}\eta_j
&=\sum_{j=0}^{2^{k_0}-1}\eta_j+\sum_{k=k_0}^{m-1}\sum_{j=2^k}^{2^{k+1}-1}\eta_j
\\&\leq \sum_{j=0}^{2^{k_0}-1}\eta_j + \sum_{k=k_0}^{m-1}\frac{512\sqrt{2^{k}}\phi(2^{k+1})}{\log(2^{k})}
\\&\leq \sum_{j=0}^{2^{k_0}-1}\eta_j + \frac{512\phi(T)}{\log(2)}\cdot\underset{\leq 8\cdot\frac{2^{m/2}}{m}=\frac{8\sqrt{T}}{\log_2(T)}}{\underbrace{\sum_{k=k_0}^{m-1}\frac{2^{k/2}}{k}}}
\\&\leq \frac{2^{20}\phi(T)\sqrt{T}}{\log(T)}~.
\end{align*}

On the other hand, Lemma~\ref{lem: 1d bound 2} gives for all sufficiently large $T$,\footnote{We may assume that $T$ is sufficiently large so that $\sum_{j=0}^{T}\eta_j>\half$.
Otherwise, if $\sum_{j=0}^{\infty}\eta_j\leq\half$, Theorem~\ref{thm: main} is trivial since GD does not even converge for $f(x)=x,~\Xcal=[-1,1],~x_0=1$, let alone at the prescribed rate.}
\[
\sum_{j=0}^{T-1}\eta_j\geq \frac{1}{4e^2\Ecal_\bmeta(T)}
\geq \frac{\sqrt{T}}{4e^2\phi(T)}~.
\]
Overall, we see that for arbitrarily large $T$ it holds that
\begin{align*}
&\frac{2^{20}\phi(T)\sqrt{T}}{\log(T)}
\geq  \frac{\sqrt{T}}{4e^2\phi(T)}\\
\implies~~~&\phi(T)\geq 2^{-10}\sqrt{\log(T)}~,
\end{align*}
contradicting $\phi(T)=o(\sqrt{\log(T)})$ and completing the proof.

\section{Discussion} \label{sec: conclusion}

In this work, we proved that for any stepsize schedule not depending on the horizon $T$, the optimal anytime convergence rate is $\Theta(\sqrt{\log(T)/T})$, for the last iterate of both GD and SGD. In particular, the information-theoretically optimal rate $\Theta(1/\sqrt{T})$ cannot be achieved unless the stepsize schedule depends on $T$. As discussed, this proves the conjecture of \citet{jain2019making} for SGD, which thus far remained open, and even extends the result to GD.

Our work leaves open several follow-up questions. 
%
First, it would be of interest to address the strongly-convex case, where there exist similar gaps between the anytime $\log T/T$ last iterate rate \citep{shamir2013stochastic} as opposed to $1/T$ when $T$ is chosen in advance \citep{jain2019making}. We believe our technique should be useful in order to resolve this setting as well.
Second, 
we note that a stepsize sequence can have an optimally converging \emph{sub}-sequence with $(t_k)_{k=0}^{\infty}\subsetneq \NN$ satisfying $\Ecal^{G,D}_\bmeta(t_k)=O(DG/\sqrt{t_k})$. For example, applying a ``doubling trick'' by concatenating optimal time-dependent stepsize sequences (e.g., as provided by \citealt{jain2019making}) for increasing powers of $2$, yields an infinite sequence $\bmeta$ so that $\Ecal^{G,D}_\bmeta(2^k)=O(DG/\sqrt{2^{k}/2})=O(DG/\sqrt{2^{k}})$.
    However, this leads to an exponentially sparse subset of optimal stopping times, and in particular the set of optimal stopping times have density zero.\footnote{A set $\mathcal{T}\subset \NN$ is called a zero-density set if $\lim_{T\to\infty}\frac{|\mathcal{T}\cap \{1,\dots,T\}|}{T}=0$.}
    It is therefore interesting to ask how ``dense'' suboptimal stopping times can be, and whether the sparsity exhibited by the doubling trick is inevitable: Is it true that for any stepsize sequence, any sub-sequence of stopping times which converge  optimally necessarily has density zero? Notably, proving this claim would strengthen our result from showing that the set of suboptimal stopping times is nonempty, to showing that it has positive density.

  %


\bibliographystyle{plainnat}
\bibliography{bib}

\newpage

\appendix

\section{Proof of Lemma~\ref{lem: high dim bound}}

To prove Lemma~\ref{lem: high dim bound}, we start by proving the following auxiliary result.

\begin{lemma} \label{lem: aj bj}
    Let $T\in\NN$, and suppose $(a_j)_{j=0}^{T}, (b_j)_{j=0}^{T}\geq0$ are non-negative sequences satisfying the following conditions:
\begin{enumerate}[(i)]
    \item $\sum_{j=0}^{T}a_j^2\leq \half$.
    \item $\sum_{j=0}^{T-1}(\eta_j b_j)^2\leq \frac{1}{4}$.
    \item For all $j$ it holds that $b_j\leq\half$.
    \item For all $j$ it holds that $a_j \sum_{k=j+1}^{T}\eta_k\leq \half\eta_j b_j$.
\end{enumerate}
Then $\Ecal_\bmeta(T)\geq \half\sum_{j=0}^{T-1}a_j b_j \eta_j$.
\end{lemma}

\begin{proof}[Proof of Lemma~\ref{lem: aj bj}]
For $i\in\{0,\dots,T\}$ we define the vectors
\[
\bv_i:=\sum_{j=1}^{i}a_{j-1}\be_{j}-b_i\be_{i+1}\in\reals^{T+1}~,
\]
and let
$f:\reals^{T+1}\to\reals$ be the function
\[
f(\bx):=\max_{i}\{\bv_i^\top \bx\}~.
\]
Note that each component $\bx\mapsto \bv_i^\top \bx$ is $1$-Lipschitz since $\|\bv_i\|^2\leq \sum_{j=0}^{T-1}a_j^2+\max_{j}b_j^2\leq 1$, and therefore
$f$ is $1$-Lipschitz as the maximum of $1$-Lipschitz functions.
We consider GD applied to $f$,
initialized at the origin and projected onto the unit ball $\BB\subset\reals^{T+1}$, where the given subgradient at a point $\bx\in\reals^{T+1}$ corresponds to $\bv_i$ for the minimal index $i$ such that $f(\bx)=\bv_i^\top\bx$.

We will first prove by induction over $t$ that for all $t\in[T]:$
\begin{equation} \label{eq: x_t,j}
\bx_t=\sum_{j=1}^{t}\Big(b_{j-1}\eta_{j-1}-a_{j-1}\textstyle\sum_{k=j}^{t-1}\eta_k\Big)\be_j
~.
\end{equation}
For the base case $t=1$, note that all the $\bv_i$'s are in the sub-gradient set of $\bx_0=\bzero$, so the sub-gradient choice of the minimal index leads to
\[
\bx_1=\Proj_\BB(\bzero-\eta_0\bv_0)
=\Proj_\BB(\eta_0 b_0\be_1)
\overset{(\eta_0 b_0\leq\frac{1}{2\sqrt{T+1}}<1)}{=}\eta_0 b_0\be_1~,
\]
and
therefore \eqref{eq: x_t,j} holds for $t=1$.
Now assume \eqref{eq: x_t,j} holds at time $t$.
To obtain the claim for $t+1$, we start by showing that the returned sub-gradient at $\bx_t$ is $\bv_t$. To see that, we see that for any $i\in\{0,\dots,T\}:$
\begin{align*}
\bv_i^\top \bx_t
&=
\left(\sum_{j=1}^{i}a_{j-1}\be_{j}-b_i\be_{i+1}\right)^\top \left(\sum_{j=1}^{t}\Big(b_{j-1}\eta_{j-1}-a_{j-1}\textstyle\sum_{k=j}^{t-1}\eta_k\Big)\be_j\right)
\\&=\sum_{j=1}^{\min\{i,t\}}a_{j-1}
\underset{(\star)}{\underbrace{\Big(b_{j-1}\eta_{j-1}-a_{j-1}\textstyle\sum_{k=j}^{t-1}\eta_k\Big)}}
-\one_{\{i+1\leq t\}}\cdot b_i
\underset{(\star\star)}{\underbrace{\Big(b_{i}\eta_{i}-a_{i}\textstyle\sum_{k=i+1}^{t-1}\eta_k\Big)}}~,
\end{align*}
and note that $(\star),(\star\star)>0$ by assumption \textit{(iv)}, and therefore the minimal index that maximizes the expression above is clearly $i= t$. Hence, the given sub-gradient at $\bx_t$ is 
$\bv_t$ as claimed, which gives that
\begin{align}
\bx_{t+1}&=\Proj_\BB(\bx_t-\eta_t\bv_t) \nonumber
\\
&=\Proj_\BB\left(
\sum_{j=1}^{t}\Big(b_{j-1}\eta_{j-1}-a_{j-1} \sum_{k=j}^{t-1}\eta_k\Big)\be_j
-\eta_t\sum_{j=1}^{t}a_{j-1}\be_{j}+\eta_t b_t\be_{t+1}
\right)
\nonumber \\
&=\Proj_\BB\left(\sum_{j=1}^{t+1}\Big(b_{j-1}\eta_{j-1}-a_{j-1}\textstyle\sum_{k=j}^{t}\eta_k\Big)\be_j
\right)
\nonumber \\
&=\sum_{j=1}^{t+1}\underset{(\diamond_j)}{\underbrace{\Big(b_{j-1}\eta_{j-1}-a_{j-1}\textstyle\sum_{k=j}^{t}\eta_k\Big)}}\be_j
\label{eq: t+1 induc}
\end{align}
where the last equality follows from noting that for each $j$, by assumption \textit{(iv)} it holds that
\[
b_{j-1}\eta_{j-1}
\geq (\diamond_j)
\geq b_{j-1}\eta_{j-1}-a_{j-1}\textstyle\sum_{k=j}^{T}\eta_k 
\geq \half b_{j-1}\eta_{j-1}
\geq 0~,
\]
hence, using assumption \textit{(ii)},
\[
\sum_{j=1}^{t+1}(\diamond_j)^2\leq \sum_{j=0}^{T-1}(b_{j}\eta_{j})^2
\leq \frac{1}{4}<1~.
\]
We see that \eqref{eq: t+1 induc}
completes the induction step, proving \eqref{eq: x_t,j}.

To complete the proof, we note that $\min_{\bx\in\BB}f(\bx)\leq f(\bzero)=0$, and 
recall we saw that $f(\bx_T)=\bv_T^\top\bx_T$,
therefore
\begin{align*}
\Ecal_\bmeta(T)&\geq
f(\bx_{T})-\min_{\bx\in\BB}f(\bx)
\\&\geq f(\bx_{T})
\\&=
\bv_{T}^\top
\bx_T
\\
&=\left(\sum_{j=1}^{T}a_{j-1}\be_{j}-b_T\be_{T+1}\right)^\top
\left(\sum_{j=1}^{T}\Big(b_{j-1}\eta_{j-1}-a_{j-1}\textstyle\sum_{k=j}^{T-1}\eta_k\Big)\be_j\right)
\\&=\sum_{j=1}^{T}a_{j-1}
\Big(b_{j-1}\eta_{j-1}-a_{j-1}\textstyle\sum_{k=j}^{T-1}\eta_k\Big)
\\&\geq_{\textit{(iv)}}\half\sum_{j=1}^{T}a_{j-1}b_{j-1}\eta_{j-1}
\\&=\half\sum_{j=0}^{T-1}a_{j}b_{j}\eta_{j}~.
\end{align*}

\end{proof}

\begin{proof}[Proof of Lemma~\ref{lem: high dim bound}]
    
Let $S:=\sum_{j=t_1}^{t_2}\eta_j$, and assume $S>0$ since otherwise the claim is trivial. Given $t_1\leq t\leq t_2+1$ and any $t_1\leq j<t$, let
\begin{align*}
    b_j&\equiv b:=\frac{1}{2\sqrt{1+\delta S}}~,
    \\a_j&:=\frac{b \eta_j}{2(\delta+\sum_{k=j}^{t-1}\eta_k)}=\frac{\eta_j}{4\sqrt{1+\delta S}(\delta+\sum_{k=j}^{t-1}\eta_k)}
\end{align*}
and set $a_j=b_j=0$ for al other $j\in\{0,\dots, t\}$.
We see that the conditions in Lemma~\ref{lem: aj bj} hold with respect to $T=t:$
\begin{enumerate}[(i)]
    \item It holds that
\begin{align*}
    \sum_{j=0}^{t}a_j^2&= \sum_{j=t_1}^{t-1}\frac{\eta_j^2}{16(1+\delta S)(\delta+\sum_{k=j}^{t-1}\eta_k)^2}
    \\&\leq \frac{\delta}{16}\sum_{j=t_1}^{t-1}\frac{\eta_j}{(\delta+\sum_{k=j}^{t-1}\eta_k)^2}
    \\&\leq \frac{\delta}{16}\sum_{j=t_1}^{t-1}\frac{\eta_j}{(\delta+\sum_{k=j+1}^{t-1}\eta_k)(\delta+\sum_{k=j}^{t-1}\eta_k)}
    \\&= \frac{\delta}{16}\sum_{j=t_1}^{t-1}\left(\frac{1}{\delta+\sum_{k=j+1}^{t-1}\eta_k}-\frac{1}{\delta+\sum_{k=j}^{t-1}\eta_k}\right)
    \\&=\frac{\delta}{16}\left(\frac{1}{\delta}-\frac{1}{\delta+\sum_{k=t_1}^{t-1}\eta_k}\right)
    \\&\leq \frac{1}{16}<\half~.
\end{align*}

\item It holds that
\begin{align*}
\sum_{j=0}^{t-1}(\eta_j b_j)^2
=\frac{\sum_{j=t_1}^{t-1}\eta_j^2}{4(1+\delta S)}
\leq\frac{\overset{\leq \delta}{\overbrace{({\max}_{t_1\leq j\leq t-1}\eta_j)}}\cdot\overset{=S}{\overbrace{{\textstyle\sum}_{j=t_1}^{t-1}\eta_j}}}{4(1+\delta S)}
\leq\frac{1}{4}~.
\end{align*}

\item It holds that $b_j\in\{\frac{1}{2\sqrt{1+\delta S}},0\}$ and in either case $b_j\leq\half$.

\item For any $t_1\leq j <t$ it holds that
\begin{align*}
a_j \sum_{k=j+1}^{t-1}\eta_k
=\frac{\eta_j b_j}{2}
\cdot\frac{\sum_{k=j+1}^{t-1}\eta_k}{\delta+\sum_{k=j}^{t-1}\eta_k}\leq \half\eta_j b_j~,
\end{align*}
and the condition is trivial for remaining $j$.

\end{enumerate}

Therefore, as the conditions of Lemma~\ref{lem: aj bj} hold, we apply it and get that
\begin{align}\label{Et inner bound}
\Ecal_\bmeta(t)
&\geq \half\sum_{j=0}^{t-1}a_j b_j \eta_j
=\frac{1}{16(1+\delta S)}\sum_{j=t_1}^{t-1}\frac{\eta_j^2}{\delta+\sum_{k=j}^{t-1}\eta_k}~.
\end{align}

Multiplying the inequality above by $\eta_{t-1}$, summing over $t$ and changing the order of summation gives
\begin{align*}
16S(1+\delta S)\max_{t_1< t\leq t_2+1}\Ecal_\bmeta(t)
&\geq 16(1+\delta S)\sum_{t=t_1+1}^{t_2+1}\Ecal_\bmeta(t)\eta_{t-1}
\\&\geq\sum_{j=t_1}^{t_2}\eta_j^2\sum_{m=j}^{t_2}\frac{\eta_m}{\delta+\sum_{k=j}^{m}\eta_k}
\geq \sum_{j=t_1}^{t_2}\eta_j^2\log\left(1+\frac{\sum_{k=j}^{t_2}\eta_k}{2\delta}\right)~,
\end{align*}
where the last inequality follows by noting that $\eta_m\leq \delta$ and $\log(1+x)\leq x$ imply $\frac{\eta_m}{\delta+\sum_{k=j}^{m}\eta_k}\geq \frac{\eta_m}{2\delta+\sum_{k=j}^{m-1}\eta_k}\geq \log(1+\frac{\eta_m}{2\delta+\sum_{k=j}^{m-1}\eta_k})
=\log({2\delta+\sum_{k=j}^{m}\eta_k})-\log(2\delta+\sum_{k=j}^{m-1}\eta_k)$, which telescopes when summed over $m$.

Since $\log(1+\lambda x)\geq \lambda \log(1+x)$ for all $x\geq 0,~\lambda\in[0,1]$, letting $x=\frac{S}{2\delta}$ and $\lambda=\frac{\sum_{k=j}^{t_2}\eta_k}{S}$ shows that
$\log(1+\frac{\sum_{k=j}^{t_2}\eta_k}{2\delta})\geq \frac{\sum_{k=j}^{t_2}\eta_k}{S}\log(1+\frac{S}{2\delta})$. Thus
\begin{align*}
\sum_{j=t_1}^{t_2}\eta_j^2\log\left(1+\frac{\sum_{k=j}^{t_2}\eta_k}{2\delta}\right)
&\geq \frac{\log\left(1+\frac{S}{2\delta}\right)}{S}\sum_{j=t_1}^{t_2}\eta_j^2\sum_{k=j}^{t_2}\eta_k
\\&\geq \frac{\log\left(1+\frac{S}{2\delta}\right)}{S}\cdot\frac{\left(\sum_{j=t_1}^{t_2}\eta_j\sum_{k=j}^{t_2}\eta_k\right)^2}{\sum_{j=t_1}^{t_2}\sum_{k=j}^{t_2}\eta_k}
\\&\geq \frac{\log\left(1+\frac{S}{2\delta}\right)}{S}
\cdot\frac{S^3}{4(t_2-t_1+1)}~,
\end{align*}
where the second inequality follows by Cauchy-Schwarz, and the third since $\sum_{j=t_1}^{t_2}\eta_j\sum_{k=j}^{t_2}\eta_k=\\\half\left(S^2+\sum_{j=t_1}^{t_2}\eta_j^2\right)\geq\frac{S^2}{2}$ and $\sum_{j=t_1}^{t_2}\sum_{k=j}^{t_2}\eta_k\leq (t_2-t_1+1)S$. Overall, we obtain
\[
16(1+\delta S)\max_{t_1< t\leq t_2+1}\Ecal_\bmeta(t)
\geq \frac{S\log\left(1+\frac{S}{2\delta}\right)}{4(t_2-t_1+1)}~,
\]
which completes the proof via rearrangement.

\end{proof}

\end{document}